\documentclass[12pt, twoside]{article}

\usepackage{amsmath, amssymb, latexsym, amscd, amsthm}

\usepackage[dvips]{color}
\usepackage{graphicx}

\usepackage{enumerate, mathrsfs}

\usepackage{color}
\usepackage{xcolor}

\newtheorem{thm}{Theorem}[section]
\newtheorem{cor}[thm]{Corollary}
\newtheorem{lemma}[thm]{Lemma}

\theoremstyle{definition}
\newtheorem{defn}[thm]{Definition}

\newtheorem{example}[thm]{Example}

\newcommand{\E}{\ensuremath{{\mathbb E}}}
\newcommand{\Pro}{\ensuremath{{\mathbb P}}}

\numberwithin{equation}{section}

\begin{document}

\baselineskip=17pt

\title{Discrete time risk models with $m$-dependent random variables}

\author{Nguyen Huy Hoang and Tran Dinh Phung}

\date{}

\maketitle

\renewcommand{\thefootnote}{}



\begin{abstract}
The main purpose of the paper is to study ruin probabilities in two discrete time risk models under rates of interest, where the premiums and claims are two independent sequences of $m$-dependent random variables, and  the rate of interest is a sequence of identically distributed random variables. Our results extend the corresponding ones for independent random sequences.
\end{abstract}

\section{Introduction}
Consider the classical risk model with discrete time. The premiums $\{X_n, n \geqslant 1\}$   and claims $\{Y_n, n \geqslant 1\}$  are assumed to be two independent sequences of nonnegative, independent identically distributed (i.i.d.) random variables. Then the surplus at time $n$:
\begin{equation}
U_n =   u +\sum_{i=1}^n (X_i-Y_i),\hspace{5pt}  n\geqslant 1, \label{eq1.1}
\end{equation}
or, equivalently, the stochastic process $\{U_n, n\geqslant 1\}$ satisfies
\begin{equation*}
U_n =  U_{n-1} +X_n-Y_n,\hspace{5pt}  n\geqslant 1, 
\end{equation*}
where  $U_0 = u \geqslant 0$, $u$ is the insurer’s initial surplus. The sequence $\{S_n:=\sum_{i=1}^{n}(Y_i-X_i), n\geqslant 1\}$ forms a random walk. One of interesting quantities related to the random walk is
\begin{equation*}
\Psi(u)=\Pro\Big\{\bigcup_{n=1}^\infty (U_n <0) \Big\}= \Pro\Big\{\bigcup_{n=1}^\infty (S_n >u) \Big\}.
\end{equation*}
It is well known that $\Psi(u)$ is the ultimate ruin probability in the classical risk model \eqref{eq1.1}  with the initial surplus $u$. As in Ross \cite[Theorem 7.4.2]{Ro1996}, if $\mathbb E X_1>\mathbb E Y_1$, then the Lundberg inequality provides an exponential upper bound for the ruin probability that for any $u\geqslant 0$,
\begin{equation*}
	\Psi(u)\leqslant e^{-\theta u},
\end{equation*}
where the Lundberg adjustment coefficient $\theta$ is such that
\begin{align*}  
\mathbb E e^{\theta(Y_1-X_1)}=1.
\end{align*} 

Now we consider the following two risk model, in which the effects of timing of payments and interest on the surplus process and on the ruin probability can be included, and the interests form a sequence of nonnegative random variables $\{I_n, n \geqslant 1\}$:
\begin{align}  
&U_n=(U_{n-1}+X_n)(1+I_n)-Y_n,\hspace{5pt}  n\geqslant 1, \label{eq1.2}
\end{align} 
and
\begin{align}  
&U_n=(U_{n-1})(1+I_n)+X_n-Y_n,\hspace{5pt}  n\geqslant 1. \label{eq1.3}
\end{align} 
Then the ultimate ruin probability in the above models will be defined as follows
\begin{equation*}
	\Psi(u)=\Pro\Big\{\bigcup_{n=1}^\infty (U_n <0) \Big\}.
\end{equation*}

Cai \cite{Cai2022} showed that \eqref{eq1.2} implies to
\begin{equation}
U_n= u \prod_{i=1}^n(1+I_i)  + \sum_{i=1}^n\Big(\big (X_i(1+I_i)-Y_i\big )\prod_{j=i+1}^n(1+I_j)\Big),\hspace{5pt} n\geqslant 1,  \label{eq1.4} \nonumber
\end{equation}
and \eqref{eq1.3} is equivalent to
\begin{equation}
U_n= u \prod_{i=1}^n(1+I_i)  + \sum_{i=1}^n\Big(\big (X_i-Y_i)\prod_{j=i+1}^n(1+I_j)\Big),\hspace{5pt} n\geqslant 1, \label{eq1.5}\nonumber
\end{equation}
here and afterwards we denote $\prod_{j=a}^b(1+I_j)=1$ if $a>b$.

Models \eqref{eq1.2} and \eqref{eq1.3} are two generalizations of \eqref{eq1.1}. In the risk model \eqref{eq1.2}, the premiums are collected at the beginning of each period and the surplus process $\{U_n, n\geqslant 1\}$ is called the annuity-due risk model, while the premiums in the risk model \eqref{eq1.3} are received at the end of each period and the surplus process $\{U_n, n\geqslant 1\}$ defined in \eqref{eq1.3} is called the annuity-immediate risk model.

With respect to the Lundberg inequality for the ruin probabilities in two risk models \eqref{eq1.2} and \eqref{eq1.3}, Yang \cite{Ya1999} studied a special case of \eqref{eq1.2} when the rates of interest are identical constants. Cai \cite{Cai2022first} derived stochastic inequalities for ruin probabilities in the models \eqref{eq1.2}, \eqref{eq1.3} by using martingales and renewal recursive techniques, where the premiums $\{X_n, n \geqslant 1\}$   and claims $\{Y_n, n \geqslant 1\}$  are two independent sequences of nonnegative i.i.d. random variables, and  the rate of interest $\{I_n, n \geqslant 1\}$ is a sequence of nonnegative i.i.d.  random variables. Cai \cite{Cai2022} considered the ruin probabilities in two risk models, where the rates of interest are assumed to have a dependent autoregressive structure. The author also provided an illustrative application for the compound binomial risk model. Xu and Wang \cite{XW2006} gave upper bounds for ruin probabilities in a risk model with interest force and independent premiums and claims with Markov chain interest rate.

We next recall the concept of $m$-dependence which was introduced by Hoeffding and Robbins \cite{HR1948}. 

\begin{defn}
	Let $m$ be a fixed nonnegative integer. A sequence of random variables $\{X_n, n \geqslant 1\}$ is said to be $m$-dependent if $\sigma$-algebras $\sigma\{X_1, X_2, \ldots, X_n \}$ and $\sigma\{X_{n+k}, X_{n+k+1}, \ldots \}$ are independent for all $k \geqslant m+1$ and $n\geqslant 1$.  
\end{defn}


For a sequence of random variables, Hoeffding and Robbins \cite{HR1948} showed that 0-dependence is equivalent to independence. Also, $m$-dependence implies $n$-dependence for every nonnegative integer $n>m$.

\begin{example}
Let $m$ be a fixed nonnegative integer, and let $\{Z_n, n \geqslant 1\}$  be a sequence of independent random variables. For $n\geqslant 1$, we set
	\begin{align*}
		X_1&:=Z_1+Z_2+\cdots+ Z_m;\\
		X_2&:=Z_2+Z_3+\cdots+Z_{m+1};\\
		\vdots\\
		X_n&:=Z_n+Z_{n+1}+\cdots+Z_{m+n-1}.
	\end{align*}
Then  $\{X_n, n \geqslant 1\}$   is a sequence of  $m$-dependent random variables.
\end{example}

\begin{example}
	If $\{X_n, n \geqslant 1\}$ is a sequence of $1$-dependent random variables, then subsequences $\{X _1, X_3,  X_5, ...\}$ and $\{X_2, X_4, X_6, ...\}$ are  sequences of independent random variables.
\end{example}

The $m$-dependent model arises in general time series data with arbitrary dependence structure. We refer to Lahiri \cite[Theorem A.7]{La2003} for the IID bootstrap method of $m$-dependent data and the central limit theorem for $m$-random variables. The moving blocks bootstrap scheme invented by Künsch \cite{Ku1989} has also become well known method for $m$-dependent data and weakly dependent process. In \cite{Ho2020, HB2020}, the first author obtained upper bounds for ruin probability in the insurance continuous-time risk models with $m$-dependent claims. While in \cite{HB2024} Hoang et al. considered a reinsurance  risk model  with $m$-dependent assumptions and gave upper bounds for the ruin probability by the Martingale and  inductive methods. 

Inspired by the works mentioned above, the main purpose of the present paper is to further study  ruin probabilities in two risk models \eqref{eq1.2} and \eqref{eq1.3}, where the premiums $\{X_n, n \geqslant 1\}$   and claims $\{Y_n, n \geqslant 1\}$  are two independent sequences of nonnegative, $m$-dependent random variables, and  the rate of interest $\{I_n, n \geqslant 1\}$ is a sequence of nonnegative, identically distributed random variables.

\section{Preliminary lemmas}
In this section, we give some lemmas which will be used to prove our results.

By using properties of characteristic functions and Kac's theorem (for details, see \cite[Theorem 2.1, page 6]{AB2005}), we can get the following lemma. The proofs of the lemma are quite simple and are therefore omitted.
\begin{lemma} \label{lemma2.1}
Let two sequences of random variables $\{X_n, n \geqslant 1\}$ and $\{Y_n, n \geqslant 1\}$  be independent of one another, and let $\{ \alpha_n, n \geqslant 1\}$ and $\{\beta_n, n \geqslant 1\}$ be sequences of real numbers.

(i) If $\{X_n, n \geqslant 1\}$ and $\{Y_n, n \geqslant 1\}$  are sequences of identically distributed random variables, then the sequence $\{\alpha_n X_n+\beta_n Y_n, n \geqslant 1\}$ is identically distributed.

(ii) If $\{X_n, n \geqslant 1\}$ and $\{Y_n, n \geqslant 1\}$  are sequences of independent random variables, then the sequence $\{\alpha_n X_n+\beta_n Y_n, n \geqslant 1\}$ is independent.
\end{lemma}

The next lemma is a part of the Kolmogorov-Doob inequality (for details, see \cite[Theorem 9.1, page 501]{Gu2005}).
\begin{lemma} \label{lemma2.2} 
Let $\lambda>0$, and let $\{X_n, n \geqslant 1\}$ be a submartingale. Then 	
$$\lambda\, \mathbb P \big (\max_{1\leqslant k\leqslant n}X_k>\lambda\big)\leqslant \mathbb E X_n^+\leqslant \mathbb E |X_n|.$$
\end{lemma}

\begin{lemma}[\cite{Chow}, page 213] \label{lemma2.3}
Let the random vectors $U=\{U_1, U_2, ..., U_m\}$ and $V=\{V_1, V_2, ..., V_n\}$ on $(\Omega, \mathcal F,\Pro)$ be independent of one another, and let $f$ be a Borel function on $\mathbb{R}^m \times \mathbb{R}^n$ with $|\E f(U,V)|\leqslant \infty$. If, for $x\in \mathbb{R}^m$, 
$$ g(x)=  \begin{cases}
\E f(x,V)& \text{if} \quad  |\E f(x,V)| \leqslant  \infty; \\
	0,&otherwise,
\end{cases} $$
then $g$ is a Borel function on  $\mathbb{R}^m$  with $g(U) =\E \left\{ f(U,V) | \sigma(U)\right\} $ a.s. 
\end{lemma}

The next lemma establishes the Lundberg inequality in risk model \eqref{eq1.2}. Noting that the ultimate ruin probability in this model will be estimated as follows:
\begin{align} 
	\begin{split} 
		\Psi(u)&= \Pro\Big\{\bigcup_{n=1}^\infty (U_n <0) \Big\}\\
		&=\lim_{n\to \infty}\Pro\Big\{\sum_{i=1}^n\big ((Y_i- X_i(1+I_i))\pi^n_i\big ) >     u \prod_{i=1}^n(1+I_i) \Big\}\\
		&=\lim_{n\to \infty}\Pro\Big\{\sum_{i=1}^n\Big(\big (Y_i(1+I_i)^{-1}- X_i\big )  \prod_{j=1}^{i-1} (1+I_j)^{-1} \Big)  >u \Big\},
	\end{split} \label{eq2.1} \nonumber
\end{align} 
where
\begin{equation}
\pi^n_i= \begin{cases}
\prod_{j=i+1}^n(1+I_j)&,  \quad 1\leq i \leq n\\
1&, \quad i=n.
\end{cases}\nonumber
\end{equation}

\begin{lemma}\label{lemma2.4}
Let the premiums  $\{X_n, n \geqslant 1\}$, the claim $\{Y_n, n \geqslant 1\}$ and the rate of interest $\{I_n, n \geqslant 1\}$  be independent sequences of nonnegative i.i.d. random variables with finite expectations, and let $\Psi(u)$ be the ultimate ruin probability in the model \eqref{eq1.2}. If there exists a positive real number $R$ satisfying 
\begin{align}  
\E\left(e^{R\left(Y_1(1+I_1)^{-1}- X_1\right)}\right) \leqslant 1, \label{eq2.2}
\end{align} 
then $\{Z_n = e^{RS_n}, n\geqslant 1\}$ is a  supermartingale and  
$$ \Psi(u) \leqslant e^{-Ru},\hspace{5pt} u >0,$$ 
where
\begin{align*}  
S_n =   \sum_{i=1}^n\Big(\big (Y_i(1+I_i)^{-1}- X_i\big)  \prod_{j=1}^{i-1} (1+I_j)^{-1} \Big).
\end{align*} 
\end{lemma}

\begin{proof}
First of all, we prove that $\left(Z_n = e^{RS_n}\right)_{n\geq 1}$ is a  supermartingale,  which means
$$ \E (Z_{n+1}|Z_1, Z_2, ..., Z_n) \leq Z_n \quad (a.s).$$
In fact, since
$$ Z_n = e^{RS_n}= e^{R  \sum_{i=1}^n\left[(Y_i(1+I_i)^{-1}- X_i)  \prod_{j=1}^{i-1} (1+I_j)^{-1} \right]}  $$
we have
$$  \E (Z_{n+1}|Z_1, Z_2, ..., Z_n) = Z_n \E\left(e^{R(Y_{n+1}(1+I_{n+1})^{-1}- X_{n+1})  \prod_{j=1}^{n} (1+I_j)^{-1}  } |Z_1, Z_2, ..., Z_n\right). $$
To prove
\begin{align*}
& \E\left(e^{R(Y_{n+1}(1+I_{n+1})^{-1}- X_{n+1})  \prod_{j=1}^{n} (1+I_j)^{-1}  } |Z_1, Z_2, ..., Z_n\right)\\
&=  \E\left(e^{R(Y_{n+1}(1+I_{n+1})^{-1}- X_{n+1})  \prod_{j=1}^{n} (1+I_j)^{-1}  } |X_1, X_2, ..., X_n, Y_1, Y_2, ..., Y_n, I_1, I_2, ..., I_n\right)\leq 1.
\end{align*}
Lemma \ref{lemma2.3} is applied for the case
\begin{align*}
 &f(U, V)= e^{R(Y_{n+1}(1+I_{n+1})^{-1}- X_{n+1})  \prod_{j=1}^{n} (1+I_j)^{-1}  } \\
&U = (X_1, X_2, ..., X_n, Y_1, Y_2, ..., Y_n, I_1, I_2, ..., I_n)\\
& V =  (X_{n+1}, Y_{n+1},  I_{n+1}),
\end{align*}
where $U, V$ are mutually independent. Corresponding to the value
\begin{align*}
& (X_1=x_1, ..., X_n=x_n, Y_1=y_1, ..., Y_n=y_n, I_1=i_1, ..., I_n=i_n)\\
& u =  (x_1, x_2, ..., x_n, y_1, y_2, ..., y_n, i_1, i_2, ..., i_n)
\end{align*}
we consider the function
$$ g(u)= \E \left( e^{R(Y_{n+1}(1+I_{n+1})^{-1}- X_{n+1})  \prod_{j=1}^{n} (1+I_j)^{-1}  }\right)  $$
let
$$  r=  \prod_{j=1}^{n} (1+I_j)^{-1}\leq 1, \quad p =\frac{1}{r}\geq 1,  $$
and apply Jensen's inequality
\begin{align*}
 [g(u)]^p  &= \left[\E \left( e^{R(Y_{n+1}(1+I_{n+1})^{-1}- X_{n+1}) r  }\right) \right]^p \\
& \leq  \E \left( e^{R(Y_{n+1}(1+I_{n+1})^{-1}- X_{n+1}) r p }\right) \\
&  =  \E \left( e^{R(Y_{n+1}(1+I_{n+1})^{-1}- X_{n+1})  }\right) \\
& =  \E \left( e^{R(Y_{1}(1+I_{1})^{-1}- X_{1}) }\right)  \leq 1 
\end{align*}
we obtain
\begin{align*}
g(U)& =\E \left( e^{R(Y_{n+1}(1+I_{n+1})^{-1}- X_{n+1})  \prod_{j=1}^{n} (1+I_j)^{-1}  }  | X_1, X_2, ..., X_n, Y_1, Y_2, ..., Y_n, I_1, I_2, ..., I_n \right) \\
& \leq 1.
\end{align*}
Thus
\begin{align*}
&  \E (Z_{n+1}|Z_1, Z_2, ..., Z_n)  \\
&= Z_n \E \left( e^{R(Y_{n+1}(1+I_{n+1})^{-1}- X_{n+1})  \prod_{j=1}^{n} (1+I_j)^{-1}  }  | Z_1, Z_2, ..., Z_n \right)  \leq Z_n \quad (a.s).
\end{align*}
This implies that $\left(Z_n = e^{RS_n}\right)_{n\geq 1}$ is a non-negative supermartingale. Moreover, it is easily seen that 
\begin{align*}
 \Psi(u)  &= \lim_{n \to \infty}\Psi_n(u) = \Pro\bigg\{\bigcup_{n=1}^\infty (S_n >u) \bigg\}   = \Pro\bigg\{ \lim_{n\to \infty}  \bigcup_{k=1}^n (S_k >u) \bigg\}   \\
& = \lim_{n\to \infty} \Pro\bigg\{   \bigcup_{k=1}^n (S_k >u) \bigg\}   =  \lim_{n\to \infty}  \Pro\bigg\{   \max_{k\leq n} (S_k) >u \bigg\}  \\
& = \Pro \left(\max_{k \leq n}  e^{RS_k} > e^{Ru} \right)   \\
& =  \Pro \left(\max_{k \leq n}Z_k > e^{Ru} \right).
\end{align*}
In addition, from the assumptions of the Lemma \ref{lemma2.4} we have
$$  \E \left\{Z_1\right\} = \E \left(e^{R\left(Y_1(1+I_1)^{-1}- X_1\right)}\right)  \leq 1. $$
Therefore, from the properties of  non-negative supermartingale we obtain
$$  \Psi(u)  = \lim_{n\to \infty} \Pro \left(\max_{k \leq n}Z_k \geq e^{Ru} \right)   \leq   \lim_{n\to \infty} e^{-Ru}\E \left\{Z_1\right\}   \leq e^{-Ru}. $$
The Lemma \ref{lemma2.4} has been proven. 
\end{proof}

In the following theorem, we consider the risk model \eqref{eq1.3} and provide a variant of Lemma \ref{lemma2.4}.
\begin{lemma}\label{lemma2.5}
Let the premiums  $\{X_n, n \geqslant 1\}$, the claim $\{Y_n, n \geqslant 1\}$ and the rate of interest $\{I_n, n \geqslant 1\}$  be independent sequences of nonnegative i.i.d. random variables with finite expectations, and let $\Psi(u)$ be the ultimate ruin probability in the model \eqref{eq1.3}. If there exists a positive real number $R$ satisfying 
\begin{align}  
\E\left(e^{R\left(Y_1- X_1\right)}\right) \leqslant 1, \label{eq2.3}
\end{align} 
then $\{Z_n = e^{RS_n}, n\geqslant 1\}$ is a  supermartingale and  
$$ \Psi(u) \leqslant e^{-Ru},\hspace{5pt} u >0,$$ 
where
\begin{align*}  
S_n =   \sum_{i=1}^n\Big(\big (Y_i- X_i\big)  \prod_{j=1}^{i-1} (1+I_j)^{-1} \Big).
\end{align*} 
\end{lemma}

\begin{proof}
We first prove that $\left(Z_n = e^{RS_n}\right)_{n\geq 1}$ is a  supermartingale,  which means
$$ \E (Z_{n+1}|Z_1, Z_2, ..., Z_n) \leq Z_n \quad (a.s).$$
In fact, since
$$ Z_n = e^{RS_n}= e^{R  \sum_{i=1}^n\left[(Y_i- X_i)  \prod_{j=1}^{i} (1+I_j)^{-1} \right]}  $$
we have
$$  \E (Z_{n+1}|Z_1, Z_2, ..., Z_n) = Z_n \E\left(e^{R(Y_{n+1}- X_{n+1})  \prod_{j=1}^{n+1} (1+I_j)^{-1}  } |Z_1, Z_2, ..., Z_n\right). $$
To prove
\begin{align*}
& \E\left(e^{R(Y_{n+1}- X_{n+1})  \prod_{j=1}^{n+1} (1+I_j)^{-1}  } |Z_1, Z_2, ..., Z_n\right)\\
&=  \E\left(e^{R(Y_{n+1}- X_{n+1})  \prod_{j=1}^{n+1} (1+I_j)^{-1}  } |X_1, X_2, ..., X_n, Y_1, Y_2, ..., Y_n, I_1, I_2, ..., I_n\right)\leq 1.
\end{align*}
Lemma \ref{lemma2.3} is applied for the case
\begin{align*}
 &f(U, V)= e^{R(Y_{n+1}- X_{n+1})  \prod_{j=1}^{n+1} (1+I_j)^{-1}  } \\
&U = (X_1, X_2, ..., X_n, Y_1, Y_2, ..., Y_n, I_1, I_2, ..., I_n, I_{n+1})\\
& V =  (X_{n+1}, Y_{n+1},  I_{n+1}),
\end{align*}
where $U, V$ are mutually independent. Corresponding to the value
\begin{align*}
& (X_1=x_1, ..., X_n=x_n, Y_1=y_1, ..., Y_n=y_n, I_1=i_1, ..., I_n=i_n, I_{n+1}=i_{n+1})\\
& u =  (x_1, x_2, ..., x_n, y_1, y_2, ..., y_n, i_1, i_2, ..., i_n, i_{n+1})
\end{align*}
we consider the function
$$ g(u)= \E \left( e^{R(Y_{n+1}- X_{n+1})  \prod_{j=1}^{n+1} (1+I_j)^{-1}  }\right)  $$
let
$$  r=  \prod_{j=1}^{n+1} (1+I_j)^{-1}\leq 1, \quad p =\frac{1}{r}\geq 1,  $$
and apply Jensen's inequality
\begin{align*}
 [g(u)]^p  &= \left[\E \left( e^{R(Y_{n+1}- X_{n+1}) r  }\right) \right]^p \\
& \leq  \E \left( e^{R(Y_{n+1}- X_{n+1}) r p }\right) \\
&  =  \E \left( e^{R(Y_{n+1}- X_{n+1})  }\right) \\
& =  \E \left( e^{R(Y_{1}- X_{1}) }\right)  \leq 1 
\end{align*}
we obtain
\begin{align*}
g(U)& =\E \left( e^{R(Y_{n+1}- X_{n+1})  \prod_{j=1}^{n+1} (1+I_j)^{-1}  }  | X_1, X_2, ..., X_n, Y_1, Y_2, ..., Y_n, I_1, I_2, ..., I_n, I_{n+1} \right) \\
& \leq 1.
\end{align*}
Thus
\begin{align*}
&  \E (Z_{n+1}|Z_1, Z_2, ..., Z_n)  \\
&= Z_n \E \left( e^{R(Y_{n+1}- X_{n+1})  \prod_{j=1}^{n} (1+I_j)^{-1}  }  | Z_1, Z_2, ..., Z_n \right)  \leq Z_n \quad (a.s).
\end{align*}
This implies that $\left(Z_n = e^{RS_n}\right)_{n\geq 1}$ is a non-negative supermartingale. Moreover, it is easily seen that 
\begin{align*}
 \Psi(u)  &= \lim_{n \to \infty}\Psi_n(u) = \Pro\bigg\{\bigcup_{n=1}^\infty (S_n >u) \bigg\}   = \Pro\bigg\{ \lim_{n\to \infty}  \bigcup_{k=1}^n (S_k >u) \bigg\}   \\
& = \lim_{n\to \infty} \Pro\bigg\{   \bigcup_{k=1}^n (S_k >u) \bigg\}   =  \lim_{n\to \infty}  \Pro\bigg\{   \max_{k\leq n} (S_k) >u \bigg\}  \\
& = \Pro \left(\max_{k \leq n}  e^{RS_k} > e^{Ru} \right)   \\
& =  \Pro \left(\max_{k \leq n}Z_k > e^{Ru} \right).
\end{align*}
In addition, from the assumptions of the Lemma \ref{lemma2.5} we have
$$  \E \left\{Z_1\right\} = \E \left(e^{R\left(Y_1- X_1\right)}\right)  \leq 1. $$
Therefore, from the properties of  non-negative supermartingale we obtain
$$  \Psi(u)  = \lim_{n\to \infty} \Pro \left(\max_{k \leq n}Z_k \geq e^{Ru} \right)   \leq   \lim_{n\to \infty} e^{-Ru}\E \left\{Z_1\right\}   \leq e^{-Ru}, $$
which completes the proof. 
\end{proof}

\section{Main results}

With the preliminaries accounted for, we can now state the main results. 

The following theorem  provides the ruin probability in risk model \eqref{eq1.2}, where the premiums $\{X_n, n \geqslant 1\}$   and claims $\{Y_n, n \geqslant 1\}$  are sequences of $m$-dependent random variables.

\begin{thm}\label{thm3.1}
Let the premiums  $\{X_n, n \geqslant 1\}$ and the claim $\{Y_n, n \geqslant 1\}$ be sequences of  nonnegative,  identically distributed, $m$-dependent random variables with finite expectations, and let the rate of interest $\{I_n, n \geqslant 1\}$  be a sequence of i.i.d. nonnegative random variables with finite expectations. Suppose that the sequences  $\{X_n, n \geqslant 1\}$, $\{Y_n, n \geqslant 1\}$, $\{I_n, n \geqslant 1\}$ are mutually independent. If there exists a positive real number $R$ satisfying \eqref{eq2.2}, then the ultimate ruin probability in the model \eqref{eq1.2} satisfies
	\begin{equation}\label{H1.2} \Psi(u) \leqslant (m+1)e^{-R\frac{u}{m+1}},\end{equation}
	where  
	$$  u>  \frac{(m+1)\ln(m+1)}{R}.$$
\end{thm}

\begin{proof}
Without loss of generality, suppose that sequences of random variables  $\{X_n, n \geq 1\}$, $\{Y_n, n \geq 1\}$ are $1$-dependent. Thus, from sequence  $\{S_n \}_{n \geq 1}$  with 
$$ S_n =   \sum_{i=1}^n\bigg[(Y_i(1+I_i)^{-1}- X_i)  \prod_{j=1}^{i-1} (1+I_j)^{-1} \bigg],$$ we can form two following subsequences 

\begin{align*}
&\left(S^{(1)}_{k}\right)_{k\geq 1} : \\
&S^{(1)}_{k} = (Y_1(1+I_1)^{-1}- X_1)   + (Y_3(1+I_3)^{-1}- X_3)  \prod_{j=1}^{2} (1+I_j)^{-1} \\
& \quad \quad \quad \quad  + \cdots+  (Y_{2k-1}(1+I_{2k-1})^{-1}- X_{2k-1})  \prod_{j=1}^{2k-2} (1+I_j)^{-1} \\
\end{align*}
and
\begin{align*}
&\left(S^{(2)}_{k}\right)_{k\geq 1} : \\
&S^{(2)}_{k} =(Y_2(1+I_2)^{-1}- X_2)   + (Y_4(1+I_4)^{-1}- X_4)  \prod_{j=1}^{3} (1+I_j)^{-1} \\
& \quad \quad \quad \quad  + \cdots+  (Y_{2k}(1+I_{2k})^{-1}- X_{2k})  \prod_{j=1}^{2k-1} (1+I_j)^{-1}, 
\end{align*}
denote $Z^{(1)}_k =e^{RS^{(1)}_k }$;  $Z^{(2)}_k =e^{RS^{(2)}_k }$. We prove that $\left( Z^{(1)}_k =e^{RS^{(1)}_k }\right)_{k\geq 1}$ and $\left( Z^{(2)}_k =e^{RS^{(2)}_k }\right)_{k\geq 1}$ are supermartingales. (Note that for each $S^{(j)}_k, (j =1, 2)$  all of its elements are undependent random variables.). In fact, since	
\begin{align*}
  & Z^{(1)}_k =e^{RS^{(1)}_k } \\
&= e^{R\left[  (Y_1(1+I_1)^{-1}- X_1)   + (Y_3(1+I_3)^{-1}- X_3)  \prod_{j=1}^{2} (1+I_j)^{-1} + \cdots+  (Y_{2k-1}(1+I_{2k-1})^{-1}- X_{2k-1})  \prod_{j=1}^{2k-2} (1+I_j)^{-1}\right] }, \\
\end{align*}
$$  S^{(1)}_{k+1}   = S^{(1)}_{k}  +   (Y_{2k+1}(1+I_{2k+1})^{-1}- X_{2k+1})  \prod_{j=1}^{2k} (1+I_j)^{-1}.  $$
We obtain
$$  Z^{(1)}_{k+1} =e^{RS^{(1)}_{k+1} } = Z^{(1)}_k e^{R(Y_{2k+1}(1+I_{2k+1})^{-1}- X_{2k+1})  \prod_{j=1}^{2k} (1+I_j)^{-1}}.   $$
Now, we consider
\begin{align*}
 &\E\left(Z^{(1)}_{k+1}\big|Z^{(1)}_{1},Z^{(1)}_{2}, ..., Z^{(1)}_{k} \right)  \\
&= Z^{(1)}_{k} \E\left(e^{R(Y_{2k+1}(1+I_{2k+1})^{-1}- X_{2k+1})  \prod_{j=1}^{2k} (1+I_j)^{-1}}|Z^{(1)}_{1},Z^{(1)}_{2}, ..., Z^{(1)}_{k} \right).
\end{align*}
According to the proof of Lemma \ref{lemma2.4} we have
$$   \E\left(e^{R(Y_{2k+1}(1+I_{2k+1})^{-1}- X_{2k+1})  \prod_{j=1}^{2k} (1+I_j)^{-1}}\big|Z^{(1)}_{1}, Z^{(1)}_{2}, ..., Z^{(1)}_{k} \right)  \leq 1.   $$
Thus
\begin{align*}
&   \E\left(Z^{(1)}_{k+1}\big|Z^{(1)}_{1},Z^{(1)}_{2},..., Z^{(1)}_{k} \right)  \\
& = Z^{(1)}_{k} \E\left(e^{R(Y_{2k+1}(1+I_{2k+1})^{-1}- X_{2k+1})  \prod_{j=1}^{2k} (1+I_j)^{-1}}|Z^{(1)}_{1},Z^{(1)}_{2}, ...,Z^{(1)}_{k} \right) \leq Z^{(1)}_{k}.\\  
\end{align*}
Hence, $\left( Z^{(1)}_k =e^{RS^{(1)}_k }\right)_{k\geq 1}$  is supermartingale. By a similar proof, we can obtain that  $\left( Z^{(2)}_k =e^{RS^{(2)}_k }\right)_{k\geq 1}$  is  supermartingale.

We have
\begin{equation}\label{eqP}
\Pro \bigg\{\bigcup_{k=1}^n (S_k >u) \bigg\} \leq    \Pro  \left\{ \max_{k\leq n} S^{(1)}_{k} > \frac{u}{2}  \right\} +  \Pro  \left\{ \max_{k\leq n} S^{(2)}_{k} > \frac{u}{2}  \right\}.
\end{equation}

Now we have
\begin{align*}
\Pro  \left\{ \max_{k\leq n} S^{(1)}_{k} > \frac{u}{2}  \right\}  & =  \Pro  \left\{ \max_{k\leq n}  e^{RS^{(1)}_{k}}  > e^{R \frac{u}{2}}  \right\} \\
&  =  \Pro  \left\{ \max_{k\leq n} Z^{(1)}_{k} \geq   e^{R \frac{u}{2} }  \right\} \\
& \leq  e^{-R\frac{u}{2}}\E \left\{Z_1\right\} \leq e^{-R\frac{u}{2}}
\end{align*}
 by   $\left( Z^{(1)}_k =e^{RS^{(1)}_k }\right)_{k\geq 1}$ is non-negative supermartingale and   $$\E \left\{Z_1\right\} = \E \left(e^{R\left(Y_1(1+I_1)^{-1}- X_1\right)}\right)  \leq 1.$$
Thus, we obtain

\begin{equation}\label{eqP1}
 \Pro  \left\{ \max_{k\leq n} S^{(1)}_{k} > \frac{u}{2}  \right\} \leq e^{-R\frac{u}{2}}.
\end{equation}
Similarly prove we have
\begin{equation}\label{eqP2}
\Pro  \left\{ \max_{k\leq n} S^{(2)}_{k} > \frac{u}{2}  \right\}   \leq  e^{-R\frac{u}{2}}. 
\end{equation}
By  \eqref{eqP}, \eqref{eqP1} and \eqref{eqP2}  we obtain
$$  \Pro \bigg\{\bigcup_{k=1}^n (S_k >u) \bigg\} \leq    \Pro  \left\{ \max_{k\leq n} S^{(1)}_{k} > \frac{u}{2}  \right\} +  \Pro  \left\{ \max_{k\leq n} S^{(2)}_{k} > \frac{u}{2}  \right\} \leq   2e^{-R\frac{u}{2}} $$
and therefore
\begin{align*}
 \Psi(u)&=  \Pro \bigg\{\bigcup_{n=1}^\infty (S_n >u) \bigg\} =  \Pro \bigg\{  \lim_{n \to \infty}\bigcup_{k=1}^n (S_k >u) \bigg\}\\
&  = \lim_{n \to \infty} \Pro \bigg\{\bigcup_{k=1}^n (S_k >u) \bigg\} \leq    2e^{-R\frac{u}{2}}.
\end{align*}
Using the above proof technique for the cases that $\{X_n, n \geq 1\}$ and  $\{Y_n, n \geq 1\}$ are sequences of  $m$-dependent random variables, we obtain 
$$    \Psi(u) \leq (m+1)e^{-R\frac{u}{m+1}}.   $$
The Theorem \ref{thm3.1} has been proven. 
\end{proof}

In the special case when $m=0$, Theorem \ref{thm3.1} implies the following corollary.
\begin{cor}\label{cor3.2}
Let the premiums  $\{X_n, n \geqslant 1\}$ and the claim $\{Y_n, n \geqslant 1\}$ be sequences of  nonnegative,  identically distributed,  independent  random variables with finite expectations, and let the rate of interest $\{I_n, n \geqslant 1\}$  be a sequence of i.i.d. nonnegative random variables with finite expectations. Suppose that the sequences  $\{X_n, n \geqslant 1\}$, $\{Y_n, n \geqslant 1\}$, $\{I_n, n \geqslant 1\}$ are mutually independent. If there exists a positive real number $R$ satisfying \eqref{eq2.2}, then the ultimate ruin probability in the model \eqref{eq1.2} satisfies
	$$
	\Psi(u) \leqslant e^{-R u },
	$$
	where  $  u> 0.$
\end{cor}

The next corollary establishes the Lundberg inequality for the case the rates of interest are identical constants.
\begin{cor}\label{cor3.3}
Let the premiums  $\{X_n, n \geqslant 1\}$ and the claim $\{Y_n, n \geqslant 1\}$ be sequences of  nonnegative,  identically distributed,  independent  random variables with finite expectations. 
Suppose that the sequences  $\{X_n, n \geqslant 1\}$, $\{Y_n, n \geqslant 1\}$ are mutually independent. If  $I_n =r>0$ for all $n=1, 2, \ldots $ and there exists a positive real number $R$ satisfying 
\begin{align}  
\E\left(e^{R\left(Y_1(1+r)^{-1}- X_1\right)}\right) \leqslant 1,  \nonumber
\end{align} 
 then the ultimate ruin probability in the model \eqref{eq1.2} satisfies
	$$ \Psi(u) \leqslant e^{-R u } $$
where  $  u> 0.$
\end{cor}

In the following theorem, we obtain the ruin probability in risk model \eqref{eq1.3}.
\begin{thm}\label{thm3.4}
Let the premiums  $\{X_n, n \geqslant 1\}$ and the claim $\{Y_n, n \geqslant 1\}$ be sequences of  nonnegative,  identically distributed, $m$-dependent random variables with finite expectations, and let the rate of interest $\{I_n, n \geqslant 1\}$  be a sequence of i.i.d. nonnegative random variables with finite expectations. Suppose that the sequences  $\{X_n, n \geqslant 1\}$, $\{Y_n, n \geqslant 1\}$, $\{I_n, n \geqslant 1\}$ are mutually independent. If there exists a positive real number $R$ satisfying \eqref{eq2.3}, then the ultimate ruin probability in the model \eqref{eq1.3} satisfies
	\begin{equation}\label{H1.3}
	\Psi(u) \leqslant (m+1)e^{-R\frac{u}{m+1}},
	\end{equation}
	where  
	$$  u>  \frac{(m+1)\ln(m+1)}{R}.$$
\end{thm}

\begin{proof}
Without loss of generality, suppose that sequences of random variables  $\{X_n, n \geq 1\}$, $\{Y_n, n \geq 1\}$ are $1$-dependent. Thus, from sequence  $\{S_n \}_{n \geq 1}$  with 
$$ S_n =   \sum_{i=1}^n\bigg[(Y_i- X_i)  \prod_{j=1}^{i} (1+I_j)^{-1} \bigg],$$ we can form two following subsequences 

\begin{align*}
&\left(S^{(1)}_{k}\right)_{k\geq 1} : \\
&S^{(1)}_{k} = (Y_1- X_1)(1+I_1)^{-1}   + (Y_3- X_3)  \prod_{j=1}^{3} (1+I_j)^{-1} \\
& \quad \quad \quad \quad  + \cdots+  (Y_{2k-1}- X_{2k-1})  \prod_{j=1}^{2k-1} (1+I_j)^{-1} \\
\end{align*}
and
\begin{align*}
&\left(S^{(2)}_{k}\right)_{k\geq 1} : \\
&S^{(2)}_{k} =(Y_2- X_2)(1+I_2)^{-1}   + (Y_4- X_4)  \prod_{j=1}^{4} (1+I_j)^{-1} \\
& \quad \quad \quad \quad  + \cdots+  (Y_{2k}- X_{2k})  \prod_{j=1}^{2k} (1+I_j)^{-1}, 
\end{align*}
denote $Z^{(1)}_k =e^{RS^{(1)}_k }$;  $Z^{(2)}_k =e^{RS^{(2)}_k }$. We prove that $\left( Z^{(1)}_k =e^{RS^{(1)}_k }\right)_{k\geq 1}$ and $\left( Z^{(2)}_k =e^{RS^{(2)}_k }\right)_{k\geq 1}$ are supermartingales. (Note that for each $S^{(j)}_k, (j =1, 2)$  all of its elements are undependent random variables.). In fact, since	
\begin{align*}
  & Z^{(1)}_k =e^{RS^{(1)}_k } \\
&= e^{R\left[  (Y_1- X_1)(1+I_1)^{-1}   + (Y_3- X_3)  \prod_{j=1}^{3} (1+I_j)^{-1} + \cdots+  (Y_{2k-1}- X_{2k-1})  \prod_{j=1}^{2k-1} (1+I_j)^{-1}\right] }, \\
\end{align*}
$$  S^{(1)}_{k+1}   = S^{(1)}_{k}  +   (Y_{2k+1}- X_{2k+1})  \prod_{j=1}^{2k+1} (1+I_j)^{-1}.  $$
We obtain
$$  Z^{(1)}_{k+1} =e^{RS^{(1)}_{k+1} } = Z^{(1)}_k e^{R(Y_{2k+1}- X_{2k+1})  \prod_{j=1}^{2k+1} (1+I_j)^{-1}}.   $$
Now, we consider
\begin{align*}
 &\E\left(Z^{(1)}_{k+1}\big|Z^{(1)}_{1},Z^{(1)}_{2}, ..., Z^{(1)}_{k} \right)  \\
&= Z^{(1)}_{k} \E\left(e^{R(Y_{2k+1}- X_{2k+1})  \prod_{j=1}^{2k+1} (1+I_j)^{-1}}|Z^{(1)}_{1},Z^{(1)}_{2}, ..., Z^{(1)}_{k} \right).
\end{align*}
According to the proof of Lemma \ref{lemma2.5} we have
$$   \E\left(e^{R(Y_{2k+1}- X_{2k+1})  \prod_{j=1}^{2k+1} (1+I_j)^{-1}}\big|Z^{(1)}_{1}, Z^{(1)}_{2}, ..., Z^{(1)}_{k} \right)  \leq 1.   $$
Thus
\begin{align*}
&   \E\left(Z^{(1)}_{k+1}\big|Z^{(1)}_{1},Z^{(1)}_{2},..., Z^{(1)}_{k} \right)  \\
& = Z^{(1)}_{k} \E\left(e^{R(Y_{2k+1}- X_{2k+1})  \prod_{j=1}^{2k+1} (1+I_j)^{-1}}|Z^{(1)}_{1},Z^{(1)}_{2}, ...,Z^{(1)}_{k} \right) \leq Z^{(1)}_{k}.\\  
\end{align*}
Hence, $\left( Z^{(1)}_k =e^{RS^{(1)}_k }\right)_{k\geq 1}$  is supermartingale. By a similar proof, we can obtain that  $\left( Z^{(2)}_k =e^{RS^{(2)}_k }\right)_{k\geq 1}$  is  supermartingale.

We have
\begin{equation}\label{eq3.6}
\Pro \bigg\{\bigcup_{k=1}^n (S_k >u) \bigg\} \leq    \Pro  \left\{ \max_{k\leq n} S^{(1)}_{k} > \frac{u}{2}  \right\} +  \Pro  \left\{ \max_{k\leq n} S^{(2)}_{k} > \frac{u}{2}  \right\}.
\end{equation}

Now we have
\begin{align*}
\Pro  \left\{ \max_{k\leq n} S^{(1)}_{k} > \frac{u}{2}  \right\}  & =  \Pro  \left\{ \max_{k\leq n}  e^{RS^{(1)}_{k}}  > e^{R \frac{u}{2}}  \right\} \\
&  =  \Pro  \left\{ \max_{k\leq n} Z^{(1)}_{k} \geq   e^{R \frac{u}{2} }  \right\} \\
& \leq  e^{-R\frac{u}{2}}\E \left\{Z_1\right\} \leq e^{-R\frac{u}{2}}
\end{align*}
 by   $\left( Z^{(1)}_k =e^{RS^{(1)}_k }\right)_{k\geq 1}$ is non-negative supermartingale and  
 $$\E \left\{Z_1\right\} = \E \left(e^{R\left(Y_1- X_1\right)}\right)  \leq 1.$$
Thus, we obtain

\begin{equation}\label{eq3.7}
 \Pro  \left\{ \max_{k\leq n} S^{(1)}_{k} > \frac{u}{2}  \right\} \leq e^{-R\frac{u}{2}}.
\end{equation}
Similarly prove we have
\begin{equation}\label{eq3.8}
\Pro  \left\{ \max_{k\leq n} S^{(2)}_{k} > \frac{u}{2}  \right\}   \leq  e^{-R\frac{u}{2}}. 
\end{equation}
By  \eqref{eq3.6}, \eqref{eq3.7} and \eqref{eq3.8}  we obtain
$$  \Pro \bigg\{\bigcup_{k=1}^n (S_k >u) \bigg\} \leq    \Pro  \left\{ \max_{k\leq n} S^{(1)}_{k} > \frac{u}{2}  \right\} +  \Pro  \left\{ \max_{k\leq n} S^{(2)}_{k} > \frac{u}{2}  \right\} \leq   2e^{-R\frac{u}{2}} $$
and therefore
\begin{align*}
 \Psi(u)&=  \Pro \bigg\{\bigcup_{n=1}^\infty (S_n >u) \bigg\} =  \Pro \bigg\{  \lim_{n \to \infty}\bigcup_{k=1}^n (S_k >u) \bigg\}\\
&  = \lim_{n \to \infty} \Pro \bigg\{\bigcup_{k=1}^n (S_k >u) \bigg\} \leq    2e^{-R\frac{u}{2}}.
\end{align*}
Using the above proof technique for the cases that $\{X_n, n \geq 1\}$ and  $\{Y_n, n \geq 1\}$ are sequences of  $m$-dependent random variables, we obtain 
$$    \Psi(u) \leq (m+1)e^{-R\frac{u}{m+1}}.   $$
The Theorem \ref{thm3.4} has been proven. 
\end{proof}

In the  case  $m=0$, Theorem \ref{thm3.4} implies the following corollary.
\begin{cor}
Let the premiums  $\{X_n, n \geqslant 1\}$ and the claim $\{Y_n, n \geqslant 1\}$ be sequences of  nonnegative,  identically distributed,  independent  random variables with finite expectations, and let the rate of interest $\{I_n, n \geqslant 1\}$  be a sequence of i.i.d. nonnegative random variables with finite expectations. Suppose that the sequences  $\{X_n, n \geqslant 1\}$, $\{Y_n, n \geqslant 1\}$, $\{I_n, n \geqslant 1\}$ are mutually independent. If there exists a positive real number $R$ satisfying \eqref{eq2.2}, then the ultimate ruin probability in the model \eqref{eq1.2} satisfies
	$$
	\Psi(u) \leqslant e^{-R u },
	$$
	where  $  u> 0.$
\end{cor}

The following corollary shows the Lundberg inequality for the case the rates of interest are identical constants.
\begin{cor}
Let the premiums  $\{X_n, n \geqslant 1\}$ and the claim $\{Y_n, n \geqslant 1\}$ be sequences of  nonnegative,  identically distributed,  independent  random variables with finite expectations. 
Suppose that the sequences  $\{X_n, n \geqslant 1\}$, $\{Y_n, n \geqslant 1\}$ are mutually independent. If  $I_n =r>0$ for all $n=1, 2, \ldots $ and there exists a positive real number $R$ satisfying 
\begin{align}  
\E\left(e^{R\left(Y_1- X_1\right)}\right) \leqslant 1,  \nonumber
\end{align} 
 then the ultimate ruin probability in the model \eqref{eq1.3} satisfies
	$$ \Psi(u) \leqslant e^{-R u },$$
	where  $  u> 0.$
\end{cor}

\section{Numerical example}

In this section, we give an example for Theorem \ref{thm3.1} and Theorem \ref{thm3.4}. Suppose that $\{X_n, n \geqslant 1\}$ and $\{Y_n, n \geqslant 1\}$ are sequences of $2-$dependent random variables. Let $X_1$ has a Poisson distribution with parameter  $\lambda= 1.1$ and $Y_1$ has a gamma density with
$$ f(y) =  \frac{ \lambda_1^{\alpha_1} y^{\alpha_1-1} e^{-\lambda_1 y} }{\Gamma(\alpha_1)}, \quad   y \geq 0, $$
where $\alpha_1 = \frac{1}{2}$ and $\lambda_1 = \frac{1}{2}$.

We can easily find that the constant $R_1$ which satisfies 
$$    \E\left(e^{R_0\left(Y_1- X_1\right)}\right) =1  $$
 is $R_0 = 0.0613828$  by the Matlab program.  Let $I_n = 0.051325$ for all  $n \geq 1$.  In this case,  we can find from
$$  \E\left(e^{R_1\left(Y_1(1+I_1)^{-1}- X_1\right)}\right) = 1 $$
that $R_1 = 0.0951395$ again by the Matlab program. We have Table 1 with a range of values of $u$.

\begin{center}
Table 1. : Upper  bounds for ruin probability in the insurance discrete-time risk models \eqref{eq1.1}, \eqref{eq1.2} and \eqref{eq1.3}

\begin{tabular}{|c|c|c|c|}
\hline
 $u$ & Model \eqref{eq1.1} & Model \eqref{eq1.2} & Model \eqref{eq1.3} \\
  \hline 
 55&  0.0341836  & 0.5243464  &  0.9736296 \\
 60&  0.0251494  &  0.4474602  &  0.8789496 \\
 65&  0.0185029  &  0.3818480 &    0.7934767 \\
 70&   0.0185029 &  0.3258567 &    0.7163156 \\
 75&   0.0100152 &  0.2780755 &   0.6466580  \\
 80&   0.0073683 &  0.2373006  &   0.5837742 \\
\hline
\end{tabular}
\end{center}


Table 1 presents the values of the upper bounds given by equations \eqref{H1.2},  \eqref{H1.3}, and the Lundberg upper bound for various values of $u$. It is worth noting that the upper bounds derived in this article for risk models under $m-$dependent assumptions depend on the parameter $m$.

\vspace{5pt} 
\noindent \textbf{Acknowledgement}. We would like to thank Professor Bui Khoi Dam for helpful
discussions and valuable comments. This research is funded by University of Finance-Marketing.

\vspace{10pt} 
\noindent \textsc{Nguyen Huy Hoang}  \\ 
Faculty of Data Science,  University of Finance - Marketing, \\ Ho Chi Minh City, Vietnam\\ E-mail: hoangtoancb@ufm.edu.vn

\vspace{10pt} 
\noindent \textsc{Tran Dinh Phung}  \\ 
Faculty of Data Science,  University of Finance - Marketing, \\ Ho Chi Minh City, Vietnam\\ E-mail: td.phung@ufm.edu.vn


\begin{thebibliography}{99}
	
\bibitem{AB2005} Applebaum, David; Bhat, B. V. Rajarama; Kustermans, Johan; Lindsay, J. Martin. Quantum independent increment processes. I. From classical probability to quantum stochastic calculus. Lectures from the School ``Quantum Independent Increment Processes: Structure and Applications to Physics''. Lecture Notes in Mathematics, 1865. Springer-Verlag, Berlin, 2005. xviii+299 pp.

\bibitem{Cai2022first} J. Cai, {\it Discrete time risk models under rates of interest}, Probab. Engrg. Inform. Sci. {\bf16} (3),  309-324, 2022. 

\bibitem{Cai2022} J. Cai,  {\it Ruin probabilities with dependent rates of interest},  J. Appl. Probab. {\bf39} (2), 312-323, 2022. 	

\bibitem{Chow} 
Y.S. Chow and H. Teicher, {\it Probability Theory: Independence, Interchangeability, Martingales}, Third edition, Springer Texts in Statistics, Springer-Verlag, New York, 1997.

\bibitem{Gu2005} Gut, Allan. Probability: a graduate course. Springer Texts in Statistics. Springer, New York, 2005. xxiv+603 pp.

\bibitem{Ho2020} Nguyen Huy Hoang. Ruin probabilities for risk models with constant interest. ; translated from Ukraïn. Mat. Zh. 71 (2019), no. 10, 1430--1434 Ukrainian Math. J. 71 (2020), no. 10, 1636--1642.

\bibitem{HB2020} Hoang, Nguyen Huy; Ta, Bao Quoc. Ruin probabilities of continuous-time risk model with dependent claim sizes and interarrival times. Internat. J. Uncertain. Fuzziness Knowledge-Based Systems 28 (2020), suppl. 1, 69--80.

\bibitem{HB2024} Hoang, N. H., Ly, T. T. H., Chung, N. Q.: Inequalities for the probability of ruin in a reinsurance risk model with $m$-dependence assumptions. 
 J. Math. Inequal. 18(2), 705--717 (2024).

\bibitem{HR1948} Hoeffding, Wassily; Robbins, Herbert. The central limit theorem for dependent random variables. Duke Math. J. 15 (1948), 773--780.
	
\bibitem{Ku1989} Künsch, Hans R. The jackknife and the bootstrap for general stationary observations. Ann. Statist. 17 (1989), no. 3, 1217--1241.	
	
\bibitem{La2003}  Lahiri, S. N. Resampling methods for dependent data. Springer Series in Statistics. Springer-Verlag, New York, 2003. xiv+374 pp.
	
\bibitem{Ro1996} Ross, Sheldon M. Stochastic processes. Second edition. Wiley Series in Probability and Statistics: Probability and Statistics. John Wiley \& Sons, Inc., New York, 1996.	

\bibitem{XW2006} Xu, Lin; Wang, Rongming. Upper bounds for ruin probabilities in an autoregressive risk model with a Markov chain interest rate. J. Ind. Manag. Optim. 2 (2006), no. 2, 165--175.

\bibitem{Ya1999} Yang, Hailiang. Non-exponential bounds for ruin probability with interest effect included. Scand. Actuar. J. 1999, no. 1, 66--79. 
\end{thebibliography}
\end{document}